\documentclass{amsart}
\usepackage[numbers]{natbib}
\usepackage{geometry}
\usepackage{amssymb}
\usepackage{amsthm}
\usepackage{mathtools}
\usepackage{diagbox}
\usepackage{slashbox}
\usepackage{tikz}
\usepackage{graphicx}
\usepackage{algorithm}
\usepackage{algpseudocode}
\usepackage{multirow}
\usepackage{bm}
\usepackage{url}
\usepackage{doi}
\usepackage{hyperref}
\usepackage{cleveref}

\newcommand{\credal}{{\mathcal{P}}}
\newcommand{\expe}{{\mathbb{E}}}
\newcommand{\lexpe}{{\underline{\expe}}}
\newcommand{\uexpe}{{\overline{\expe}}}

\newcommand*\circled[1]{\tikz[baseline=(char.base)]{
\node[shape=circle,draw,inner sep=1pt] (char) {#1};}}
\newcommand*\myboxed[1]{\tikz[baseline=(char.base)]{
\node[shape=rectangle,draw,inner sep=2.5pt] (char) {#1};}}

\newtheorem{definition}{Definition}
\newtheorem{lemma}{Lemma}
\newtheorem{theorem}{Theorem}
\newtheorem{corollary}{Corollary}
\newtheorem{example}{Example}

\title{Regret-based budgeted decision rules under severe uncertainty}
\author{Nawapon Nakharutai}
\address{Department of Statistics, Faculty of Science, Chiang Mai University}
\author{S\'ebastien Destercke}
\address{UMR CNRS 7253 Heudiasyc, Universite de Technologie de Compiegne, Sorbonne Universite}
\author{Matthias C. M. Troffaes}
\address{Department of Mathematical Sciences, Durham University}

\begin{document}

\begin{abstract}
One way to make decisions under uncertainty is to select an optimal option from a possible range of options, by maximizing the expected utilities derived from a probability model. However, under severe uncertainty, identifying precise probabilities is hard. For this reason, imprecise probability models uncertainty through convex sets of probabilities, and considers decision rules that can return multiple options to reflect insufficient information. Many well-founded decision rules have been studied in the past, but none of those standard rules are able to control the number of returned alternatives. This can be a problem for large decision problems, due to the cognitive burden decision makers have to face when presented with a large number of alternatives. Our contribution proposes regret-based ideas to construct new decision rules which return a bounded number of options, where the limit on the number of options is set in advance by the decision maker as an expression of their cognitive limitation. We also study their consistency and numerical behaviour.
\end{abstract}


\maketitle


\section{Introduction}\label{sec:intro}
Consider a decision problem where a decision maker wants to choose the best option from a set of  available options, but where there are some uncertain variables that also influence the value of the decision. Such problems are routinely met in machine learning, where a classifier will typically return an uncertainty model over the possible classes (see \cite[Ch.~10]{augustin2014introduction} and \cite{quost2018classification}); in structure design, where the quality of each design depends on multiple uncertain factors~\cite{limbourg2007fault,2005:Aughenbaugh:Paredis,bains2020economic,Schumann2011}; in agronomy where the weather and soil qualities are uncertain by nature~\cite{paton2015robust}; etc. A classical way to model such a situation is to associate a reward with each option assumed to be expressed in numerical utility scale and that depends on the uncertain state of nature. In this way, we can view an uncertain reward, and thereby, each decision, as a bounded real-valued function on the set of states of nature. Such function is commonly called an act \cite{1963:Anscombe:Aumann} or a gamble \cite{Walley91} in decision theory.

If uncertainty is described by precise probabilities for all events, a classical approach is to select the act that yields the highest expected utility \cite{1963:Anscombe:Aumann}. However, such precise probabilities may not always be obtainable in a reliable way, particularly when we only have little information about the states of nature and their chance to happen. 
One way to handle this situation is to relax the need for a unique, precise probability and to consider imprecise probability theory, in which uncertainty is modelled by sets of probabilities~\cite{augustin2014introduction}. This theory includes as special cases many commonly used uncertainty models, such as precise probabilities, possibility distributions, and belief functions. In this paper, we assume that our  knowledge about the true state of nature is represented by a closed convex set of probability mass functions.

Using sets of probabilities then requires to generalize the classical expected utility criterion. This is usually done either by proposing decision rules that return a single\footnote{Or multiple indifferent options.} option as optimal such as $\Gamma$-maximin, $\Gamma$-maximax or Hurwicz, or decision rules that can return multiple incomparable options as optimal, such as interval dominance, maximality and E-admissibility \cite{2007:Troffaes}. Nevertheless, these criteria can return any subset of options, from a single optimal act to the whole set of acts \cite{2020:nak}. This may be problematic, as it is conceivable that even if decision makers are keen to consider multiple options, they still may wish to limit the number of returned options, for example because of cognitive or monetary constraints. Such requirements are natural in many settings, such as set-valued classification~\cite{lorieul2021classification,denis2017confidence} or recommendation problems~\cite{viappiani2020equivalence}.

However, such \emph{budgeted} decision rule that limits the number of returned decisions still remains to be defined and explored within the imprecise probabilistic setting. This paper aims to propose and study such rules, using a regret criterion and expressing the rule either as a minimax or maximin problem. More precisely, we define the value of a given subset (of limited size) as the regret one would feel if an adversary was to pick an alternative from outside the retained subset. Note that we already studied the minimax version of this rule in a previous conference paper~\cite{k:budget2022}. This paper provides extended proofs, examples and numerical simulation of this minimax criterion, and introduces the maximin version, which was not studied before. 

The paper is organised as follows. In \cref{sec:preliminaries}, we present necessary notations and basic concepts for regret-based budgeted decision rules. We devote \cref{sec:minimax} to review regret-based budgeted decision rules by recalling the minimax criterion from our previous study. In \cref{sec:maximin}, we introduce the new maximin criterion and study their properties and consistency with classical imprecise probability decision criteria and provide an algorithm for the maximin criterion. In \cref{sec:computation}, we perform some computational experiments of these two regret-based budgeted decision rules. \Cref{sec:illustrations} provides two illustrative applications of the proposed rules. Finally, \cref{sec:conclusion} concludes the paper.

\section{Preliminaries and definitions}\label{sec:preliminaries}
We denote by $\Omega$ a finite set of possible states of nature about which we are uncertain.  We will consider decision problems where a subject can choose a finite number of acts from a set of acts $A$. For each act $a \in A$, if $\omega\in\Omega$  turns out the be the true state of nature, then $a(\omega)$ will represent the reward obtained by the subject (an end-user, a decision maker, \dots). We assume that the subject can specify a utility for each reward. Therefore, an act $a$ can be viewed as a real-valued function on $\Omega$ representing an uncertain reward expressed on a utility scale. 

In order to select an act, the subject can take their beliefs about the true state of nature into account. If their beliefs can be expressed through a probability mass function $p$, then they can simply select an act that maximizes their expectation $\expe_p(a) \coloneqq \sum_{\omega \in \Omega} p(\omega)a(\omega)$. However, when information is lacking, they might only be able to represent their knowledge about the unknown true value $\omega$ via a closed convex set of probability mass functions on $\Omega$.  Such a set $\credal$  of probability mass functions is called a credal set \cite{levi1980a}.

Let $\mathcal{A}$ denote the set of all finite non-empty sets of acts. Mathematically, we can then define:
\begin{definition}
A \emph{decision rule} is a function $D\colon\mathcal{A} \to \mathcal{A}$ such that $D(A)\subseteq A$ for every $A\in\mathcal{A}$.
\end{definition}
In the case of a precise probability mass function, the classical decision rule is $D(A)=\arg\max_{a \in A} \expe_p(a)$.
Given a credal set $\credal$, there are several decision rules that extend expected utility. Some can return a single act, e.g. $\Gamma$-maximin (or $\Gamma$-maximax) which maximizes the worst case (or minimal) expected utility of acts in $A$, and others a set of multiple acts, e.g. interval dominance and maximality \cite{2007:Troffaes}. In this study, we will compare our budgeted decision rules to maximality, which is induced by a partial ordering as follows: for any acts $a$ and $a'$, we say that an act $a$ dominates $a'$, denoted $a \succ a'$, whenever $\lexpe(a - a') > 0$, where $\lexpe(f)\coloneqq \min_{p \in \credal} \expe_p(f)$ is the lower bound of the expectation of a function $f$ taken over all probabilities of $\credal$; observe that $a-a'$ is the point-wise difference between two acts, and is therefore a real-valued function over $\Omega$. Note that $\lexpe(a - a') > 0$ is equivalent to requiring that $\expe_p(a) > \expe_p(a')$ for all $p \in \credal$, and is therefore equivalent to a robust version of the classical expectation-based decision rule where $a$ dominates $a'$ if $\expe_p(a) > \expe_p(a')$. 

The set of maximal acts in $A$ with respect to $\succ$ is then defined by:
\begin{equation}
D_M(A)\coloneqq\{a \in A \colon \nexists a' \in A \text{ s.t. }  a' \succ a\}
\end{equation}
which contains all undominated acts in $A$ with respect to $\succ$. In other words, $a$ is a maximal element in $A$ if and only if 
\begin{equation}
\min_{a'\in A} \uexpe(a - a') \geq 0,
\end{equation}
where $\uexpe(f)\coloneqq \max_{p \in \credal} \expe_p(f)$ is defined as the upper expectation operator, which is dual to the lower expectation since $\lexpe(f)=-\uexpe(-f)$. Note that maximality can return any non-empty subset of $A$, from a single act up to the whole set $A$~\cite{2020:nak}. If most or all elements of $A$ are returned by maximality, this may not be helpful to a decision maker, especially when $A$ is very large or when checking alternatives is costly. 

To address this issue, we introduce a so-called
budgeted decision rule.
\begin{definition}
A decision rule $D$ is said to be \emph{$k$-budgeted}
if $|D(A)|\le k$ for all $A\in\mathcal{A}$.
\end{definition}
In the above, $|D(A)|$ denotes the \emph{cardinality} of $D(A)$ (i.e. the number of elements of $D(A)$). The parameter $k$ is therefore an upper bound on the number of returned alternatives. In practice, its value can be settled accordingly to the application constraints, i.e., the number of alternatives that can be realistically inspected, for instance by a decision maker or by a second automatic processing step. It may also depends on the cardinality $|\mathcal{A}|$, e.g., passing on at most $10\%$ of $\mathcal{A}$ to the next processing step.

Next, we consider whether a $k$-budgeted decision rule preserves maximality. Specifically, we define consistency properties with respect to maximality as follows:

\begin{definition}
A decision rule $D$ is said to be \emph{strongly consistent} (with $D_M$) if for all $A\in\mathcal{A}$, we have
\begin{equation}
D(A) \subseteq D_M(A)
\end{equation}
\end{definition}

\begin{definition}
A decision rule $D$ is said to be \emph{weakly consistent} (with $D_M$) if for all $A\in\mathcal{A}$
\begin{equation}
D(A) \cap D_M(A) \neq \emptyset
\end{equation}
\end{definition}

Strong consistency ensures that a rule selects maximal acts only. Weak consistency ensures that 
it selects at least one maximal act, though it may also select non-maximal acts.
Clearly, strong consistency implies weak consistency. One can easily adapt those definitions to other decision rules possibly returning multiple elements, but we will focus here on maximality.

\begin{example}\label{ex:Intro_IP}
Consider the state space $\Omega=\{\omega_1,\omega_2,\omega_3\}$ and the acts of \cref{tab:acts_reg_mML}. Suppose furthermore that the credal set is specified as 
\begin{equation}\label{eq:credal}
\credal = \{p \in \mathbb{P} \colon p(\omega_3) \leq p(\omega_1), \ p(\omega_3)\leq 0.3\}
\end{equation}
where $\mathbb{P}$ denotes the set of all probability mass functions on $\Omega$.
The corresponding credal set is displayed in \cref{fig:simplex}.

\begin{figure}
	\begin{center}
	\begin{tikzpicture}
	\coordinate (x1) at (90:3cm);
	\coordinate (x3) at (210:3cm);
	\coordinate (x2) at (-30:3cm);
	\node[above] at (x1) {$p(\omega_1)$}  ;
	\node[right] at (x2) {$p(\omega_2)$};
	\node[left] at (x3) {$p(\omega_3)$};
	
	\draw [thick] (x1.south) -- (x2.north east) -- (x3.north west) -- (x1.south);
    \draw[dotted,very thick] (barycentric cs:x1=-0.15,x2=0.85 ,x3=0.3) -- (barycentric cs:x1=0.85,x2=-0.15 ,x3=0.3) node[left] {\small $p(\omega_3)= 0.3$}  ;
        \draw[dotted, very thick] (barycentric cs:x1=0.6,x2=-0.2,x3=0.6) node[left] {\small $p(\omega_3) = p(\omega_1)$} -- (barycentric cs:x1=-0.1,x2=1.2 ,x3=-0.1)   ;
        \path[fill=green,opacity=0.5] (barycentric cs:x1=1,x2=0 ,x3=0) -- (barycentric cs:x1=0,x2=1 ,x3=0) -- (barycentric cs:x1=0.3,x2=0.4 ,x3=0.3) -- (barycentric cs:x1=0.7,x2=0 ,x3=0.3) -- cycle;
	\end{tikzpicture}
	\end{center}
	\caption{Credal set of \cref{ex:Intro_IP} in barycentric coordinates (in green)}
			\label{fig:simplex}
\end{figure}
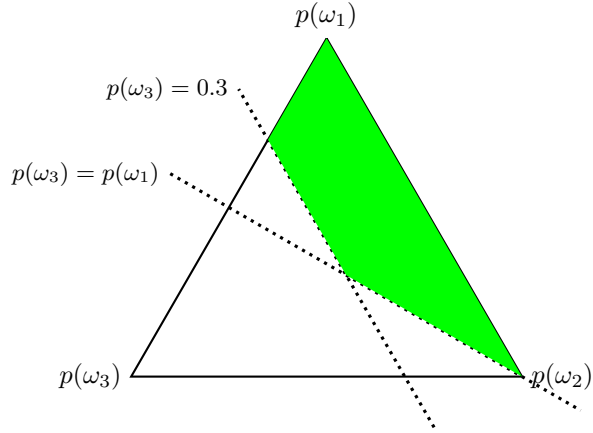

\begin{table}
\begin{minipage}{0.36\textwidth}
	\begin{center}
\begin{tabular}{c|ccc}
 & $\omega_1$ & $\omega_2$ & $\omega_3$ \\ \hline
$a_1$ & 6 & 3 & 1 \\
$a_2$ & 2 & 7 & 4 \\
$a_3$ & 5 & 1 & 8 \\
$a_4$ & 5 & 4 & 3 \\
$a_5$ & 1 & 2 & 6 \\ \hline
\end{tabular}
\vspace{0.5em}
\caption{Acts of \cref{ex:Intro_IP}}\label{tab:acts_reg_mML}
	\end{center}
 \end{minipage}
	\hfill
 \begin{minipage}{0.63\textwidth}
 \begin{center}
\begin{tabular}{|c|*{5}{c}|}\hline
\backslashbox{$j$}{$i$}
&\makebox[2em]{1}&\makebox[2em]{2}&\makebox[2em]{3}
&\makebox[2em]{4}&\makebox[2em]{5}\\\hline
	1 & - & 4.0 & 2.0 & 1.0 & 5.0 \\ 
	2 & 4.0 & - & 6.0 & 3.0 & 5.0 \\
	3 & 1.4 & 3.3 & - & 1.5 & 4.0 \\
	4 & 1.0 & 3.0 & 3.0 & - & 4.0 \\
	5 & $-0.4$ & $-0.1$ & 1.0 &$-1.1$  & - \\ \hline
\end{tabular}
\vspace{0.5em}
\caption{Values of $\uexpe(a_j - a_i)$}\label{table:example:mML}
\end{center}
\end{minipage}
\end{table}

\Cref{table:example:mML} gives the values of $\uexpe(a_j - a_i)=-\lexpe(a_i-a_j)$ where $j\neq i$. The set of maximal elements of $A$ is $D_M(A)=\{a_1,a_2,a_3,a_4\}$, as only row 5 of \cref{table:example:mML} contains negative values. For instance, we have that $a_5-a_1=(-5,-1,5)$, and the upper expected value $\uexpe(a_5 - a_1)=-0.4$ is obtained by considering the distribution $p(\omega_1)=0.3,p(\omega_2)=0.4,p(\omega_3)=0.3$ (the extreme point of the credal set that is the intersection of the dotted lines in \cref{fig:simplex}), which gives $\uexpe(a_5 - a_1)=-5 \times 0.3 + (-1) \times 0.4 + 5 \times 0.3=-0.4$. 
\end{example}

In the next section, we will construct $k$-budgeted decision rules based on the idea of regret. We introduce first the notion of regret we will consider in our decision rules, before proposing and studying its minimax and maximin versions to obtain recommendations with a limited budget. 

\section{Minimax criterion}\label{sec:minimax}
\subsection{Regret}

Let $a$ and $a'$ be acts in $A$. Then $\uexpe(a' - a)=\max_{p \in \credal} \expe_p(a' - a)= -\lexpe(a-a')$ is the maximal expected gain if we exchange $a'$ for $a$, in other words, the worst possible expected loss, or regret, to us if we keep $a$ instead of replacing $a$ for $a'$. If $a$ has a better expected utility than $a'$ under all $p\in\mathcal{P}$, then this value will be negative. 
Consider an act $a$ and any set of acts $S'$ which does not contain $a$.
The maximal loss, or worst reget, of not replacing $a$ for any act $a'\in S'$ is therefore
\begin{equation}\label{eq:single_regret_loss2}
ML(a,S')\coloneqq\max_{a' \in S'} \uexpe(a' - a).
\end{equation}
where we take the maximum to be $-\infty$ if $S'=\emptyset$.
If we had to pick a single alternative from $A$, a natural choice would be to minimize our worst regret given by \cref{eq:single_regret_loss2} for $S'=A\setminus\{a\}$.

\subsection{Review minimax criterion}

In this section, we recall the main results from \cite{k:budget2022}, with some additional elements and illustrations here for the sake of completeness, and to enable comparison to the maximin approach that will be introduced later.

We now consider a set-valued criterion based on \cref{eq:single_regret_loss2}. Consider any set $S$ of acts such that $\emptyset\neq S \subseteq A$. In this first minimax scenario, we first have to choose an element $a$ within the set $S$, and for each $a \in S$, the opponent can then pick $a'\in S'\coloneqq A\setminus S$ yielding the highest gain to them (the maximal loss to us), 
$\max_{a'\in S'}\uexpe(a'-a)$. The minimax regret of choosing a given set $S$ is then~\cite[eq.~(3)]{k:budget2022}:
\begin{equation}\label{eq:k-regret-loss}
mML(S,A)
\coloneqq
\min_{a \in S}ML(a, S')
=\min_{a \in S}\max_{a'\in S'}\uexpe(a'-a).
\end{equation}
Note that $mML(A,A)=-\infty$, because for $S=A$ we have $S'=\emptyset$. 

Next, we recall some basic properties of $mML(S,A)$.
First, \cref{eq:k-regret-loss} is monotone with respect to set inclusion (the bigger the set we can select from, the lower our loss)~\cite[Lemma 1]{k:budget2022}:
\begin{lemma}\label{lem:monotone}
For any $\emptyset\neq S\subseteq T \subseteq A\in\mathcal{A}$, we have that $mML(S,A) \geq mML(T,A)$.
\end{lemma}
\begin{proof}
Recall $S'\coloneqq A\setminus S$ and $T'\coloneqq A\setminus T$.
Indeed, 
\begin{align}
mML(T,A) & = \min_{a \in T} ML(a, T')   \\
& \leq \min_{a \in S} ML(a, T')   &&(\text{since } S \subseteq T) \\
& = \min_{a \in S} \max_{a' \in T'} \uexpe(a' - a) \\
& \leq \min_{a \in S} \max_{a' \in S'} \uexpe(a' - a)&& (\text{since } T' \subseteq S')\\
& = mML(S,A)
\end{align}
\end{proof}
Next, if $mML(S,A)$ is negative, then $S$ is a superset of the set of maximal elements~\cite[Theorem 1]{k:budget2022}:
\begin{theorem}\label{thm:1} For $\emptyset\neq S \subseteq A$, 
$mML(S,A)<0$
if and only if
there is an $a\in S$ such that for all $a'\in S'\coloneqq A\setminus S$ we have that $a\succ a'$.
\end{theorem}
\begin{proof}
By the definition, we have
\begin{align}
mML(S,A) < 0
&  \iff  \min_{a \in S}\max_{a' \in S'}\uexpe(a' - a) < 0\\
& \iff \exists a \in S,\ \max_{a' \in S'} \uexpe(a' - a) < 0 \\
& \iff \exists a \in S,\ \forall a' \in S',\ \uexpe(a' - a) < 0 \\
& \iff \exists a \in S,\ \forall a' \in S',\ \lexpe(a - a') > 0
\end{align}
\end{proof}

\begin{corollary}\label{cor:weak-consistency-negative}
If $mML(S,A) < 0$ then $D_M(A) \subseteq S$.
\end{corollary}
\begin{proof}
By \cref{thm:1}, if $mML(S,A) < 0$ then all $a'\in S'$ are dominated by some $a\in S$, and therefore, no $a'\in S'$ can be maximal. Consequently, it must be that $S$ contains all maximal elements.
\end{proof}

Given the fact that a decision maker wants to choose $S$ so as to minimize its regret, an optimal subset $S \subseteq A$ with respect to the $mML$ criterion can be defined as
\begin{equation}\label{eq:S_optim_minmax}
S^*_k(A)\coloneqq\arg\min_{\substack{\emptyset\neq S\subseteq A\\ |S|\le k}} mML(S, A).
\end{equation}
As the $\arg\min$ may not be unique, but may contain sets of alternatives that are indifferent from a regret viewpoint, we assume one set is picked at random if need be. The minimum value is denoted by
\begin{equation}\label{eq:S_optim_minimax_value}
mML_k^*(A)\coloneqq mML(S^*_k(A), A).
\end{equation}
By \cref{lem:monotone}, $S_k^*(A)$ will always be either $A$ if $|A|\le k$, or a set containing exactly $k$ elements if $|A|>k$,
and by \cref{cor:weak-consistency-negative},
$D_M(A)\subseteq S_k^*(A)$ if $mML^*_k(A) < 0$.
In particular, we know that $D_M(A)\subseteq S_k^*(A)$ if $mML^*_k(A) < 0$ or $|A|\le k$.
Therefore, it makes sense to define an $mML$ budgeted decision rule as follows:
\begin{equation}\label{eq:mML}
D^k_{mML}(A)\coloneqq\begin{cases} D_M(A) &  \text{ if } mML^*_k(A) < 0  \text{ or } |A|\le k\\ S^*_k(A) & \text{ otherwise}\end{cases}
\end{equation}

In \cite{k:budget2022}, we provided an example to show that $S^*_k(A)$ is not unique, and that we can have $S^*_k(A) \not\subseteq S^*_{k+1}(A)$. The latter fact shows in particular that the solution of \cref{eq:S_optim_minmax} cannot always be obtained incrementally by gradually increasing $k$, i.e., that a greedy approach will generally be sub-optimal.

Fortunately, there is an efficient algorithm \cite[Algorithm 1]{k:budget2022} to obtain $S^*_k(A)$ and $mML^*_k(A)$ without evaluating all possible sets $S$ of size $k$, of which there are $\binom{|A|}{k}$. This algorithm is represented in \cref{alg:S^*_k}. Note that this algorithm is polynomial. Loop 2-4 requires computing a quadratic number of upper expectations (which can be done by linear programming), and loop 6-10 is a sorting procedure.

\begin{algorithm}
\caption{Finding $S^*_k(A)$}\label{alg:S^*_k}
\begin{algorithmic}[1]
\Require $A = \{a_1,a_2,\dots, a_n\}$, $\credal$, $k$
\Ensure $S^*_k(A)$, $mML^*_k(A)$
\For{$i = 1\colon n$}
\For{$j = 1\colon n$, $j\neq i$}
\State compute $e_{ij}\coloneqq\uexpe(a_j - a_i)$
\EndFor
\EndFor
\For{$i = 1\colon n$}
\State $S[i] \gets$ set such that $\{e_{ij}\colon j\in S[i]\}$ are the $k$ largest elements of $\{e_{ij}\colon j\neq i\}$
\State $M[i]\gets \min_{j\in S[i]}e_{ij}$
\State $J[i]\gets \arg\min_{j\in S[i]}e_{ij}$
\EndFor
\State $i^*\gets \arg\min_{i=1}^n M[i]$
\State \Return $\{a_j\colon j\in \{i^*\}\cup S[i^*]\setminus\{J[i^*]\}\}$, $M[i^*]$
\end{algorithmic}
\end{algorithm}

The next example continues \cref{ex:Intro_IP}, illustrating the new decision rule and the algorithm to obtain $S_k^*(A)$~\cite[Example 1]{k:budget2022}. It also illustrates a case where $S^*_k(A) \not\subseteq S^*_{k+1}(A)$. 

\begin{example}\label{ex:mML}
Consider the credal set given by \cref{ex:Intro_IP}, then the optimal sets given by \cref{eq:S_optim_minmax} and obtainable through \cref{alg:S^*_k} are
\begin{itemize}
    \item $S^*_1(A) = \{a_4\}$ with $mML^*_1(A) = 3$, 
    \item $S^*_2(A) = \{a_1,a_2\}$ with $mML^*_2(A) = 1.4$, 
    \item $S^*_3(A) = \{a_1,a_2,a_3\}$ or $\{a_2,a_3,a_4\}$ with $mML^*_3(A) = 1$ and 
    \item $S^*_4(A) = \{a_1,a_2,a_3,a_4\}$ with $mML^*_4(A) = -1.1$.
    \end{itemize}
Note that $S^*_1(A) \not\subseteq S^*_2(A)$. Let us detail how \cref{alg:S^*_k} works to get $S^*_2(A)$. Line 7 gives $S[1]=\{2,3\}$ as the two largest elements of column 1 in \cref{table:example:mML} are in 4.0 (row $2$) and 1.4 (row $3$). For the other columns, we get $S[2]=\{1,3\}$, $S[3]=\{2,4\}$, $S[4]=\{2,3\}$, and $S[5]=\{1,2\}$. We then get $M[1]=1.4$, $M[2]=3.3$, $M[3]=3$, $M[4]=1.5$, and $M[5]=5$ (of which the minimum is $1.4$ at $i^*=1$), as well as $J[1]=3$, $J[2]=3$, $J[3]=4$, $J[4]=3$, and $J[5]=1$. From this, we obtain $S^*_2(A)=\{1\} \cup \{\{2,3\} \setminus \{3\}\}$ and the corresponding mML value $1.4$. In the case of $S^*_3(A)$, the rule would return one of the two sets $\{a_1,a_2,a_3\}$ or $\{a_2,a_3,a_4\}$, picking one at random.
\end{example}

We previously gave some properties of the $mML(S,A)$ function. Next are some properties of $S^*_k(A)$~\cite[Theorem 2]{k:budget2022}:
\begin{theorem}\label{lem:weak_consist_reg_mML}
For all $k\ge 1$,
\begin{equation}
    S^*_k(A)\cap D_M(A)\neq\emptyset.
\end{equation}
\end{theorem}
\begin{proof}
For brevity, define $S\coloneqq S^*_k(A)$ and $S'\coloneqq A\setminus S^*_k(A)$.
If $|A|\le k$ then $S=A$ by \cref{lem:monotone} and therefore $S\cap D_M(A)=D_M(A)\neq\emptyset$.
If $mML(S,A)<0$ then the statement follows from \cref{cor:weak-consistency-negative}.
So, suppose that $|A|>k$ and $mML(S,A)\ge 0$. Let
\begin{align}
\label{eq:proof:ai*def}
a_{i^*}
& \coloneqq \arg \min_{a_i\in S}\max_{a_j \in S'}\uexpe(a_j-a_i),
\\
a_{j^*}&\coloneqq \arg \max_{a_j\in S'}\uexpe(a_j-a_{i^*}). \\
\intertext{Note that}
    0\le mML(S,A)&=\min_{a_i\in S}\max_{a_j \in S'}\uexpe(a_j-a_i)=\uexpe(a_{j^*}-a_{i^*}),
\end{align}
and it follows that
\begin{align}
\label{eq:a_j*i*-nonneg}
&\uexpe(a_{j^*}-a_{i^*}) \ge 0, \\
\label{eq:a_i*_k}
\forall a_j \in S',\ &
\uexpe(a_{j}-a_{i^*})\leq \uexpe(a_{j^*}-a_{i^*}) \\
\label{eq:proof:foralliexistj}
\forall a_i \in A,\ 
\forall j \in S[i],\ &
\uexpe(a_{j}-a_{i}) \geq \uexpe(a_{j^*}-a_{i^*})
\end{align}
where $S[i]$ is defined as in \cref{alg:S^*_k}.
\Cref{eq:proof:foralliexistj} holds because,
from the algorithm,
we know that
\begin{equation}
    M[i]=\min_{j\in S[i]}\uexpe(a_{j}-a_{i})
    \ge 
    M[i^*]
    =\uexpe(a_{j^*}-a_{i^*})
\end{equation}

We have now everything in place to show that $a_{i^*}$ is maximal, i.e. $\uexpe(a_{i^*}-a_{\ell}) \geq 0$ for all $a_{\ell}\in A$.
Fix any $a_{\ell}\in A$
and consider the set
\begin{equation}
\label{eq:proof:setofgoodacts}
B\coloneqq\{a_m\colon \uexpe(a_m - a_{\ell}) \geq \uexpe(a_{j^*}-a_{i^*})\}
\end{equation}
This set has at least $k$ elements by \cref{eq:proof:foralliexistj} and the definition of $S[i]$.
If $a_{i^*}\in B$, then we are done, by \cref{eq:a_j*i*-nonneg}.
Otherwise, $B$ must contain at least one element outside of $S$ and thus in $S'$, since $S$ has exactly $k$ elements and $a_{i^*}\in S$.
Choose $a_m\in B\cap S'$.
Then
\begin{align*}
\uexpe(a_{i^*}-a_{\ell}) & \geq \uexpe(a_m-a_{\ell}) - \uexpe(a_m-a_{i^*}) \\
& = \underbrace{\uexpe(a_m-a_{\ell}) - \uexpe(a_{j^*}-a_{i^*})}_{\text{non-negative by \cref{eq:proof:setofgoodacts}}} + \underbrace{\uexpe(a_{j^*}-a_{i^*})- \uexpe(a_m-a_{i^*})}_{\text{non-negative by \cref{eq:a_i*_k}}}  \geq 0.
\end{align*}
and thus, in this case,
the desired inequality also holds.
\end{proof}

From this theorem follow two simple corollaries~\cite[Corollaries 2 and 3]{k:budget2022}:
\begin{corollary}\label{cor:S_k_weekly}
For all $k\ge 1$, $S^*_k$ and $D^k_{mML}$ are weakly consistent with $D_M$. 
\end{corollary}
\begin{corollary}\label{cor:S_k_strongly}
$S^*_1$ and $D^1_{mML}$ are strongly consistent with $D_M$. 
\end{corollary}

We conclude this section by noting that \cref{alg:S^*_k} gives us an efficient way to compute the decision rule $D^k_{mML}$, which is weakly consistent with the usual maximality criterion by \cref{lem:weak_consist_reg_mML}. However, we will see in the empirical experiment that the observed consistency is indeed kind of weak, and in any case weaker than for the maximin criterion, studied in the next section.

\section{Maximin criterion}\label{sec:maximin}

\subsection{Definition}

In the previous formulation of the problem, the minimax criterion is expressed as the fact that given a subset $S$ to us, for each possible $a\in S$ we select, the adversary is then free to choose $a'$ among the remaining options $S'=A \setminus S$ for which it will bring the highest possible gain to him (or the highest possible loss to us). It then makes sense to look for the set that minimizes this possible loss.  

Conversely, given a subset $S \subseteq A$, we could consider that the adversary first picks any action within $S'\coloneqq A\setminus S$, knowing that we are then free to choose any action within $S$ to limit our loss as much as possible. This is expressed as a maximin criterion as follows:
\begin{equation}\label{eq:k-maxmin-regret}
MmL(S,A)\coloneqq\max_{a'\in S'}\min_{a \in S}\uexpe(a'-a).
\end{equation}
Note that $MmL(A,A)=-\infty$ because for $S=A$ we have $S'=\emptyset$,
and as before we take the maximum over the empty set to be $-\infty$.
For any $S\subseteq A$, the following relation holds:
\begin{equation}\label{eq:MmL<=mML}
\max_{a'\in S'}\min_{a \in S}\uexpe(a'-a) \leq \min_{a \in S}\max_{a'\in S'}\uexpe(a'-a).
\end{equation}
Thus, $MmL(S,A)\leq mML(S,A)$. We will show further that this inequality can be strict, 
but first we will explore some properties of $MmL(S,A)$.

\subsection{Properties of \texorpdfstring{$MmL(S,A)$}{MmL(S,A)}}

We can show that the maximin criterion satisfies properties that are very similar to those of the minimax criterion.

\begin{lemma}\label{lem:MmL:monotone}
For any $\emptyset\neq S\subseteq T \subseteq A\in\mathcal{A}$, we have that $$MmL(S,A) \geq MmL(T,A).$$
\end{lemma}
\begin{proof}
Recall $S'\coloneqq A\setminus S$
and $T'\coloneqq A\setminus T$.
Then
\begin{align}
MmL(T,A) & = \max_{a' \in T'}\min_{a\in T}\uexpe(a'-a) \\
& \leq \max_{a' \in T'}\min_{a\in S}\uexpe(a'-a)    &&(\text{since } S \subseteq T) \\
& \leq \max_{a' \in S'}\min_{a\in S}\uexpe(a'-a)    &&(\text{since } T' \subseteq S') \\
& = MmL(S,A).
\end{align}
\end{proof}
The next property shows that  $MmL(S,A)$ is negative if and only if every element in $S'$ is dominated by some maximal element in $S$. This is similar but not quite the same as the corresponding property for the minimax criterion in \cref{thm:1}: the quantifiers are swapped.

\begin{theorem}\label{thm:MmL:negative}
$MmL(S,A)<0$
if and only if
for all $a' \in S':=A \setminus S$, there is an $a \in S$ such that $a\succ a'$.
\end{theorem}
\begin{proof}
By the definition, we have
\begin{align*}
MmL(S,A) < 0 &  \iff \max_{a'\in S'} \min_{a \in S}\uexpe(a' - a) < 0\\
& \iff \forall a' \in S',\ \exists a \in S,\ \uexpe(a' - a) < 0 \\
& \iff \forall a' \in S',\ \exists a \in S,\ \lexpe(a - a') > 0
\end{align*}
\end{proof}
So, contrary to \cref{thm:1}, the act $a \in S$ dominating a given $a' \in S'$ in \cref{thm:MmL:negative} can be different for each $a'$. This has an important practical impact, as will be confirmed in further experiments. \Cref{thm:MmL:negative} also leads to the following corollary.
\begin{corollary}\label{cor:MmL:consistency:negative}
If $MmL(S,A) < 0$ then $D_M(A) \subseteq S$.
\end{corollary}
\begin{proof}
By \cref{thm:MmL:negative}, each $a' \in S'$ is dominated by a maximal element in $S$. Therefore, $S'$ contains no maximal elements, so consequently all maximal elements must be in $S$.
\end{proof}

Similar to the minimax criterion, we can define an optimal maximin subset of a given size as \begin{equation}\label{eq:Maximin_opti_set}S^+_k(A)\coloneqq\arg\min_{\substack{\emptyset\neq S\subseteq A\\ |S|\le k}} MmL(S, A)
\end{equation}
and we denote the minimum value by
\begin{equation}\label{eq:Maximin_opti_set_value}
MmL^+_k(A)\coloneqq MmL(S^+_k(A), A).
\end{equation}
Note that $S^+_k(A)$ is also not guaranteed to be unique. If this occurs, we assume one set is picked at random, as these are indifferent sets. This can be employed as a budgeted decision rule as follows:
\begin{equation}\label{eq:MmL}
D^k_{MmL}(A)\coloneqq\begin{cases} D_M(A) &  \text{ if } MmL^+_k(A) < 0 \text{ or } |A|\le k\\ S^+_k(A) & \text{ otherwise}.\end{cases}
\end{equation}

\subsection{Computing \texorpdfstring{$S^+_k(A)$}{S+}}

Let us now deal with the problem of computing $S^+_k(A)$. A first question that arises is then to know whether $S^+_k(A)$ and $S^*_k(A)$ do not coincide in general, in which case one could use \cref{alg:S^*_k} to solve the problem in polynomial time. Let us first give the solution for our previous example.

\begin{example}\label{ex:MmL1}
Consider the same situation as in \cref{ex:Intro_IP}. For this situation and various values of $k$, we have that
\begin{itemize}
    \item $S^+_1(A) = \{a_4\}$ with $MmL(S^+_1,A) = 3.0$, 
    \item $S^+_2(A) = \{a_1,a_2\}$ with $MmL(S^+_2,A) = 1.4$, 
    \item $S^+_3(A) = \{a_1,a_2,a_3\}$ or $\{a_2,a_3,a_4\}$ with $MmL(S^+_3,A) = 1.0$ and 
    \item $S^+_4(A) = \{a_1,a_2,a_3,a_4\}$ with $MmL(S^+_4,A) = -1.1$,
    \end{itemize}
Again, in the case of $S^+_3(A)$, the rule would return one of the two sets, picking one at random.
\end{example}

In the above example, we get the same results for $S^+_k(A)$ and values $MmL^+_k(A)$ as for the minimax criterion. This shows, among other things, that $S^+_k(A) \not\subseteq S^+_{k+1}(A)$ and that a greedy approach applied to maximin will generally result in sub-optimal solutions. The next example however shows that such equalities do not hold in general, and experiments will later confirm that this is often the case. 

\begin{example}\label{ex:MmL2}
Suppose that we have the same space $\Omega$, the credal set specified by \cref{eq:credal} with the set of acts listed in \cref{tab:acts_for_MmL}.
\begin{table}
\begin{minipage}{0.36\textwidth}
	\begin{center}
\begin{tabular}{c|ccc}
 & $\omega_1$ & $\omega_2$ & $\omega_3$ \\ \hline
$a_1$ & 6 & 4 & 2 \\
$a_2$ & 7 & 1 & 4 \\
$a_3$ & 10 & 4 & 8 \\
$a_4$ & 2 & 7 & 2 \\
$a_5$ & 7 & 1 & 9 \\ 
$a_6$ & 7 & 8 & 2 \\ \hline
\end{tabular}
\vspace{0.5em}
\caption{Acts of \cref{ex:MmL2}}\label{tab:acts_for_MmL}
	\end{center}
\end{minipage}
\hfill
\begin{minipage}{0.63\textwidth}
\begin{center}
\begin{tabular}{|c|*{6}{c}|}\hline
\backslashbox{$j$}{$i$}
&\makebox[2em]{1}&\makebox[2em]{2}&\makebox[2em]{3}
&\makebox[2em]{4}&\makebox[2em]{5}&\makebox[2em]{6}\\\hline
	1 & - & 3.0 & 0.0 & 4.0 & 3.0 & $-0.7$ \\ 
	2 & 1.3 & - & $-3.0$ & 5.0 & 0.0 & 0.6\\
	3 & 4.6 & 3.3 & - & 8.0 & 3.0 & 3.9 \\
	4 & 3.0 & 6.0 & 3.0 & - & 6.0 & $-1.0$\\
	5 & 2.8 & 1.5 & $-1.8$ & 5.6  & - & 2.1\\ 
    6 & 4.0 & 7.0 & 4.0 & 5.0  & 7.0 & - \\ \hline
\end{tabular}
\vspace{0.5em}
\caption{Values of $\uexpe(a_j - a_i)$}\label{table:example:MmL2}
\end{center}
\end{minipage}
\end{table}

According to values $\uexpe(a_j - a_i), \forall j \neq i$ in \cref{table:example:MmL2}, we see that all maximal elements of $A$ are $D_M(A)=\{a_3,a_6\}$ (only lines $3$ and $6$ of \cref{table:example:MmL2} are non-negative for every column) and the optimal solution $S^*_k(A)$ for each $k$ are given as follows:
\begin{itemize}
    \item $S^*_1(A) = \{a_6\}$ with $mML^*_1(A) = 3.9$, 
    \item $S^*_2(A) = \{a_3,a_6\}$ with $mML^*_2(A) = 2.1$, 
    \item $S^*_3(A) = \{a_3,a_4,a_6\}$ with $mML^*_3(A) = 0$,
    \item $S^*_4(A) = \{a_1,a_3,a_4,a_6\}$ with $mML^*_4(A) = -1.8$ and 
    \item $S^*_5(A) = \{a_1,a_3,a_4,a_5,a_6\}$ with $mML^*_5(A) = -3.0$.
    \end{itemize}
For $S^+_k(A)$, we have 
\begin{itemize}
    \item $S^+_1(A) = \{a_6\}$ with $MmL^+_1(A) = 3.9$, 
    \item $S^+_2(A) = \{a_3,a_6\}$ with $MmL^+_2(A) = -0.7$, 
    \item $S^+_3(A) = \{a_1,a_3,a_6\}$ with $MmL^+_3(A) = -1.0$,
    \item $S^+_4(A) = \{a_1,a_3,a_4,a_6\}$ with $MmL^+_4(A) = -1.8$ and 
    \item $S^+_5(A) = \{a_1,a_3,a_4,a_5,a_6\}$ with $MmL^+_5(A) = -3.0$.
    \end{itemize}
\end{example}

This example confirms that $mML^*_k(A) \neq MmL^+_k(A)$ is possible, even when $S^*_k(A)=S^+_k(A)$ (see $k=2$ and $k=3$ in \cref{ex:MmL2}), and that it is possible to have $S^*_k(A) \neq S^+_k(A)$ (see $k=3$ in \cref{ex:MmL2}). This confirms that our previous algorithmic solution cannot be applied to find $S^+_k(A)$ in general. Notable exceptions where the two criteria will always coincide are when $k=1$ and $k\ge |A|-1$, as show the next two results.

\begin{theorem}\label{thm:S^+=S^*}
    For $k=1$ and all $k\ge |A|-1$ we have that
    \begin{equation}
        S^*_k(A)=S^+_k(A)
        \text{ and } mML^*_k(A)=MmL^+_k(A).
    \end{equation}
\end{theorem}
\begin{proof}
For $k\ge |A|$,
by \cref{lem:monotone,lem:MmL:monotone},
\begin{equation}
S^*_k(A)=S^+_k(A)=A\text{ and } mML^*_k(A)=MmL^+_k(A)=-\infty.
\end{equation}

For $k=1$, it holds because 
\begin{align*}
S^*_1(A) &= \arg\min_{S \in \mathcal{A}_1 } mML(S, A)\\
&= \arg\min_{a\in A}\max_{a' \in A\setminus \{a\}}\uexpe(a' - a)\\
& = \arg\min_{a\in A}\max_{a' \in A\setminus \{a\}}\min_{a''\in \{a\}}\uexpe(a'-a'') \\
& = \arg\min_{S \in \mathcal{A}_1 } MmL(S, A) \\
& = S^+_1(A)
\end{align*}
and note that the minima are achieved at the same values
in each of the above steps so $mML^*_1(A)=MmL^+_1(A)$.

For $k=|A|-1$, it holds because (with $n\coloneqq |A|$)
\begin{align*}
S^*_{n-1}(A) &= \arg\min_{S \in \mathcal{A}_{n-1} } mML(S, A)\\
& = \arg\min_{S \in \mathcal{A}_{n-1}}\left(\min_{a\in S}\max_{a' \in A \setminus S}\uexpe(a' - a)\right)\\
\intertext{and now, writing $S\in\mathcal{A}_{n-1}$ as $S=A\setminus\{a'\}$ for $a'\in A$,}
& = \arg\min_{a'\in A}\left(\min_{a\in A\setminus \{a'\}}\max_{a'' \in A\setminus(A\setminus \{a'\})}\uexpe(a'' - a)\right) \\
& = \arg\min_{a' \in A}\left(\min_{a\in A\setminus \{a'\}}\uexpe(a' - a)\right),
\end{align*}
and 
\begin{align*}
S^+_{n-1}(A)
& = \arg\min_{S \in \mathcal{A}_{n-1} } MmL(S, A) \\
&=  \arg\min_{S \in \mathcal{A}_{n-1}}\left(\max_{a' \in A \setminus S}\min_{a\in S}\uexpe(a' - a)\right)\\
\intertext{and again, writing $S\in\mathcal{A}_{n-1}$ as $S=A\setminus\{a'\}$ for $a'\in A$,}
&= \arg\min_{a' \in A}\left(\max_{a'' \in A\setminus(A\setminus \{a'\})}\min_{a\in A\setminus \{a'\}} \uexpe(a'' - a)\right)\\ 
& = \arg\min_{a' \in A} \left(\min_{a\in A\setminus \{a'\}}\uexpe(a' - a)\right).
\end{align*}
and again note that the minima are achieved at the same values
in each of the above steps so $mML^*_{n-1}(A)=MmL^+_{n-1}(A)$.
\end{proof}

Let us now consider the problem of finding $S^+_k(A)$ as well as the corresponding $MmL^+_k(A)$. In order to explain our algorithm, we first consider another question: given a value $\alpha$, can we find an $S \subseteq A$ of size $k$ such that $MmL(S,A) \leq \alpha$? Finding $S^+_k(A)$ then amounts to find the lowest $\alpha$ for which the answer is yes. Denoting as before by $e_{ij}\coloneqq\uexpe(a_{j}-a_{i})$ the upper expectation, let us first notice that $MmL(S,A)$ can only take as value negative infinity (if $S=A$), or one of the finite values $e_{ij}$.

Assume we choose $\alpha=e_{i'j'}$.
We should then put within $S$ elements $a_i$ such that those elements can be responses to the adversary choices leading to losses $e_{ij}$ lower than $\alpha$. The set $C_\alpha[i]$ of adversary choices for which a given element $a_i$ is an adequate response is given as 
\begin{equation}
C_\alpha[i]\coloneqq\{j\colon e_{ij} \leq \alpha\}.
\end{equation}
Consider for instance \cref{table:example:MmL2}, the value $\alpha=-1.0$ and $i=3$, then $C_{-1.0}[3]=\{2,5\}$ corresponding to the values $-3$ ($e_{32}$) and $-1.8$ ($e_{35}$). This means that if $a_3 \in S$, then it is a response to adversary choices $a_2,a_5$ (if those are not in the set $S$) leading to a value lower than $\alpha=-1.0$. Note that we can do this operation for any $i$ and any $\alpha$. Given this, there is a subset $S$ having $k$ elements leading to $MmL(S,A)\leq \alpha$ only if we can find $k$ elements whose union of sets $C_\alpha[i]$ includes all elements in $A\setminus S$, or formally if we can find a set $S$ such that
\begin{equation}\label{eq:condition_reachability}\cup_{i \in S} (C_\alpha[i]\setminus S)=A\setminus S.\end{equation}
This equation means that for every act $a_j$ in $S'=A\setminus S$ that the adversary can pick (right hand-side of the equation), there is at least one element $a_i$ in $S$ (the union of the left-hand side) we can pick such that $\uexpe(a_j - a_i) \leq \alpha$, i.e. one for which $j \in C_{\alpha}[i]$.

\begin{lemma}\label{lem:reachability-implies-maximin-upper-bound}
Let $\alpha\in\mathbb{R}\cup\{-\infty\}$,
and $\emptyset\neq S\subseteq A$.
Then \cref{eq:condition_reachability} implies
\begin{equation}
MmL(S,A)\le\alpha.
\end{equation}
\end{lemma}
\begin{proof}
For every $a\in S$, define
\begin{equation}
C'_\alpha(S,a)\coloneqq \{a'\in S'\colon \uexpe(a'-a)\le\alpha\}.
\end{equation}
By Definition,
\begin{equation}
\forall a\in S,\,\forall a'\in C'_\alpha(S,a)\colon \uexpe(a'-a)\le\alpha.
\end{equation}
Consequently,
\begin{align}
\forall a'\in \cup_{\tilde{a}\in S} C'_\alpha(S,\tilde{a}),\,\exists a \in S\colon\uexpe(a'-a)\le\alpha
\end{align}
Since $\cup_{a\in S} C'_\alpha(S,a)=S'$ by assumption, equivalently,
\begin{align}
\forall a'\in S',\,\exists a \in S\colon\uexpe(a'-a)\le\alpha
\end{align}
or equivalently,
\begin{align}
\max_{a'\in S'}\min_{a \in S}\uexpe(a'-a)\le\alpha
\end{align}
which is what we had to show, by \cref{eq:k-maxmin-regret}.
\end{proof}
\begin{lemma}
Let $1\le k\le |A|$.
If \cref{eq:condition_reachability} holds for some $\alpha\in\mathbb{R}\cup\{-\infty\}$ and $S\subseteq A$ with $|S|=k$, then $MmL^+_k(A)\le\alpha$.
\end{lemma}
\begin{proof}
Since $|S|=k$, by definition of $MmL^+_k(A)$ we have
\begin{align}
MmL^+_k(A)\le MmL(S,A)
\end{align}
Now use \cref{lem:reachability-implies-maximin-upper-bound}.
\end{proof}
\begin{lemma}
Let $1\le k\le |A|$.
Then \cref{eq:condition_reachability} holds for $\alpha=MmL^+_k(A)$ and $S=S^+_k(A)$, where $|S^+_k(A)|=k$.
\end{lemma}
\begin{proof}
That $|S^+_k(A)|=k$ follows from \cref{lem:MmL:monotone}
and \cref{eq:Maximin_opti_set}.

For brevity of notation,
throughout the rest of the proof,
we write $S$ for $S^+_k(A)$ and $\alpha$ for $MmL^+_k(A)$.
It now suffices to show that
\begin{equation}
\bigcup_{a\in S} C'_{\alpha}(S,a)=S'
\end{equation}
Since each $C'_{\alpha}(S,a)$ is a subset of $S'$, equivalently,
it suffices to show that
\begin{equation}
\forall a'\in S',\,\exists a\in S\colon a'\in C'_{\alpha}(S,a)
\end{equation}
Since $C'_{\alpha}(S,a)=\{a'\in S'\colon \uexpe(a'-a)\le \alpha\}$,
equivalently, we must show that
\begin{equation}
\forall a'\in S',\,\exists a\in S\colon \uexpe(a'-a)\le \alpha
\end{equation}
or equivalently,
\begin{equation}
\max_{a'\in S'}\min_{a\in S}\uexpe(a'-a)\le \alpha
\end{equation}
But this translates to $MmL(S,A)\le\alpha$,
and we know that $MmL(S,A)=\alpha$ by choice of $S$ and $\alpha$,
so the condition is satisfied.
\end{proof}
These last two lemmas tell us that we need to find the lowest $\alpha\in\mathbb{R}\cup\{-\infty\}$ for which there is a set $S$ of size $k$ such that \cref{eq:condition_reachability} is satisfied. While the search can be performed across all values in $\mathbb{R}\cup\{-\infty\}$, we will in practice (see \cref{alg:S^+_k} below) restrict it to the finite set of distinct values $e_{ij}$ (and $-\infty$), as $MmL(S,A)$ can only take one of these values.

Finally, note that \cref{eq:condition_reachability} can be rewritten as follows:
\begin{align}
&
\bigcup_{i \in S} (C_\alpha[i]\setminus S)=A\setminus S
\\
&\iff
\left(\bigcup_{i \in S} C_\alpha[i]\right)\setminus S=A\setminus S
\\
&\iff
\left(\bigcup_{i \in S} C_\alpha[i]\right)\cup S=A
\end{align}

The next question is: what should be a range of $\alpha$ that we are looking for? According to \cref{eq:MmL<=mML}, we know that for $|A|=n$ and $1\le k\le |A|$,
\begin{equation}
MmL^+_{k}(A)\le mML^*_{k}(A),
\end{equation}
where $mML^*_{k}(A)$ can be obtained by \cref{alg:S^*_k}, therefore, $mML^*_{k}(A)$ will be an upper bound of $\alpha$. 
For a lower bound, we notice that if we set $\alpha$ as the $(n-k)^{\text{th}}$ lowest value of all $e_{ij}$, then there are at least $n-k$ values of $e_{ij}$ that are lower or equal to $\alpha$. Thus, the corresponding set $\cup_{i \in S} (C_{\alpha}[i]\setminus S)$ may contain $n-k = |A \setminus S|$ elements for $A \setminus S$, and this is not true for any value lower than that. 

Let us illustrate this notion on our previous example, before formalizing it into an algorithm.

\begin{example}\label{ex:illu_algo}
Consider \cref{ex:MmL2}, together with the value $\alpha=-1$ and $k=2$. The obtained sets $C_{-1.0}[i]$ are pictured in \cref{table:order_value_S3}, with the $\alpha$ value circled, and the values lower than it squared. We do have $C_{-1.0}[3]=\{2,5\}$, $C_{-1.0}[6]=\{4\}$ and $C_{-1.0}[i]=\emptyset$ for all other $i$'s. At best, we have for $S=\{3,6\}$ that $C_{-1.0}[3] \cup C_{-1.0}[6]=\{2,5,4\}$ which is different from $A \setminus S=\{1,2,5,4\}$. We therefore cannot find a set $S$ of two elements such that $MmL(S,A)\leq -1.0$. Note that $-1.0$ is not the 4$^{\text{th}}$ (or $(n-k)^{\text{th}}$) lowest value of all $e_{ij}$, but it is the 3$^{\text{rd}}$ lowest value of all $e_{ij}$. Therefore, there are only three values lower or equal to $-1.0$, and the set $\cup_{i \in S} (C_{-1.0}[i]\setminus S)$ cannot include, by definition, 4 elements for $A \setminus S$ as we want. 

\begin{table}
\centering
\renewcommand{\arraystretch}{1.2}
\begin{tabular}{|c|*{6}{c}|}\hline
\backslashbox{$j$}{$i$}
&\makebox[2em]{1}&\makebox[2em]{2}&\makebox[2em]{3}
&\makebox[2em]{4}&\makebox[2em]{5}&\makebox[2em]{6}\\\hline
	1 & - & 3.0 & 0.0 & 4.0 & 3.0 & $-0.7$\\ 
	2 & 1.3 & - & \myboxed{$-3.0$} & 5.0 & 0.0 & 0.6\\
	3 & 4.6 & 3.3 & - & 8.0 & 3.0 & 3.9 \\
	4 & 3.0 & 6.0 & 3.0 & - & 6.0 & \circled{$-1.0$} \\
	5 & 2.8 & 1.5 & \myboxed{$-1.8$} & 5.6  & - & 2.1\\ 
    6 & 4.0 & 7.0 & 4.0 & 5.0  & 7.0 & - \\ \hline
\end{tabular}
\vspace{0.5em}
\caption{Representation of $C_{-1.0}[i] = \{j\colon e_{ij} \leq -1.0\}$ in \cref{ex:illu_algo}}\label{table:order_value_S3} 
\end{table}

The next value after $-1.0$ in \cref{table:example:MmL2} is $-0.7$, which is the 4$^{\text{th}}$ lowest value of all $e_{ij}$ for which there are $4$ values lower or equal to $-0.7$. Therefore, the corresponding set $\cup_{i \in S} (C_{-0.7}[i]\setminus S)$ will have a potential to be a set of 4 elements for $A \setminus S$. Specifically, the corresponding sets $C_{-0.7}[i]$ are represented in \cref{table:order_value} and are $C_{-0.7}[3]=\{2,5\}$, $C_{-0.7}[6]=\{1,4\}$ and $C_{-1.0}[i]=\emptyset$ for all other $i$'s. This time, we do have for $S=\{3,6\}$ that $C_{-0.7}[3] \cup C_{-0.7}[6]=\{1,2,5,4\}=A \setminus S$, meaning that we can find an $S$ with $MmL(S,A)\leq -0.7$, and since this is the lowest value for which we can, we have $S^+_2=\{3,6\}$.

\begin{table}
\centering
\renewcommand{\arraystretch}{1.2}
\begin{tabular}{|c|*{6}{c}|}\hline
\backslashbox{$j$}{$i$}
&\makebox[2em]{1}&\makebox[2em]{2}&\makebox[2em]{3}
&\makebox[2em]{4}&\makebox[2em]{5}&\makebox[2em]{6}\\\hline
	1 & - & 3.0 & 0.0 & 4.0 & 3.0 & \circled{$-0.7$} \\ 
	2 & 1.3 & - & \myboxed{$-3.0$} & 5.0 & 0.0 & 0.6\\
	3 & 4.6 & 3.3 & - & 8.0 & 3.0 & 3.9 \\
	4 & 3.0 & 6.0 & 3.0 & - & 6.0 & \myboxed{$-1.0$}\\
	5 & 2.8 & 1.5 & \myboxed{$-1.8$} & 5.6  & - & 2.1\\ 
    6 & 4.0 & 7.0 & 4.0 & 5.0  & 7.0 & - \\ \hline
\end{tabular}
\vspace{0.5em}
\caption{Representation of $C_{-0.7}[i]= \{j\colon e_{ij} \leq -0.7\}$ in \cref{ex:illu_algo}}\label{table:order_value}
\end{table}
\end{example}

We are now ready to put these ideas into a corresponding algorithm. \Cref{alg:S^+_k} provides the iterative strategy to find $S^+_k(A)$ and $MmL^+_k(A)$, by finding the lowest value for which the condition of \cref{eq:condition_reachability} is satisfied. \Cref{alg:reachable_alpha} simply checks that this condition can be satisfied. 

\begin{algorithm}
\caption{Finding $S^+_k(A)$}\label{alg:S^+_k}
\begin{algorithmic}[1]
\Require $A = \{a_1,a_2,\dots, a_n\}$, $\credal$, $k$
\Ensure $MmL^+_k(A)$, $S^+_k(A)$
\For{$j = 1\colon n$}
\For{$i = 1\colon n$, $i\neq j$}
\State compute $e_{ij}\coloneqq\uexpe(a_j - a_i)$
\EndFor
\EndFor
\State $S^+_k(A) \gets \emptyset$
\State $m \gets n-k$
\While {$S^+_k(A) = \emptyset$}
\State $e^\prime \gets e_{i^\prime j^\prime }$ such that $e_{i^\prime j^\prime }$ is the $m^{\text{th}}$ lowest value of all $e_{ij}$ 
\For{$i = 1\colon n$}
\State $C[i] \gets$ $C_{e'}[i]:=\{j : e_{ij} \leq e'\}.$
\EndFor
\State $S^+_k(A) \gets$ result of \cref{alg:reachable_alpha} with $C[i]$ and $e'$
\State $m \gets m+1$
\EndWhile
\State $MmL^+_k(A) \gets e'$
\State \Return $S^+_k(A), MmL^+_k(A)$
\end{algorithmic}
\end{algorithm}

\begin{algorithm}
\caption{Reachable value $\alpha$}\label{alg:reachable_alpha}
\begin{algorithmic}[1]
\Require $C_{\alpha}[i]$ for $i=1,\dots,n$, $k$
\Ensure A non-empty subset $S$ if there is a solution, $\emptyset$ if not
\State $S \gets  \emptyset$ 
\For{every subset $T \subseteq \{1,\ldots,n\}$ with $|T|=k$}
\If{$\cup_{i \in T} (C_\alpha[i]\setminus T)=A\setminus T$} 
\State $S \gets  T$ \textbf{break for}
\EndIf
\EndFor
\State \Return $S$
\end{algorithmic}
\end{algorithm}

Looking at the algorithms, it is clear that \cref{alg:reachable_alpha} has a combinatorial nature, and represents a bottleneck in our approach. The next result indicates that this part of the Algorithm is indeed a computational barrier.

\begin{theorem}
Checking whether the condition of \cref{eq:condition_reachability} can be satisfied is NP-complete.
\end{theorem}

\begin{proof}
To show this, we will show that it is equivalent to solving a dominating set problem, which is a known NP-complete problem. We can rephrase the problem we try to solve as having a set $\left[n\right]=\{1,\ldots,n\}$ of integer indices, to which are associated (possibly empty) subsets $A_i \subseteq \left[n\right]$ such that $i \not \in A_i$. For a given $k \leq n$, we want to solve the following optimisation problem:
$$\text{Find a subset } B \subseteq \left[n\right] \text{ of indices such that } $$
\begin{align}
|B|& =k\\
\cup_{i \in B} A_i &=\{1,\ldots,n\} \setminus B
\end{align}
In the above reformulation, the subset $B$ corresponds to $S$, and the subsets $A_i$ to the subsets $C_\alpha[i]$.
We can then consider the directed graph $G=(V,E)$ where the vertices are $V=\{1,\ldots,n\}$, and where there is an edge $(i,j) \in E$ whenever $j \in A_i$ ($A_i$ are the out-neighbours of $i$). Now, there is straightforward reduction of our problem to a dominating set problem, in the sense that taking $i$ in the dominating set is equivalent to taking $A_i$ in $B \Rightarrow A_i$ ``dominates" $i$ and every index in $A_i$.

So, identifying whether there is a dominating set of size $k$ in $G$ is equivalent to identify whether there is a solution to our problem, hence the two problems have the same complexity. 
\end{proof}

\Cref{fig:graph_illu} illustrates the equivalent graphs of the sets used in \cref{ex:illu_algo}. One can easily see that for $\alpha=-1.0$, there is no way to pick a cover for $k=2$, i.e., two vertices $i,j$ such that all other vertices except $i,j$ are their direct neighbours (one would need to pick at least three vertices). In contrast, this is possible in the graph corresponding to $\alpha=-0.7$.

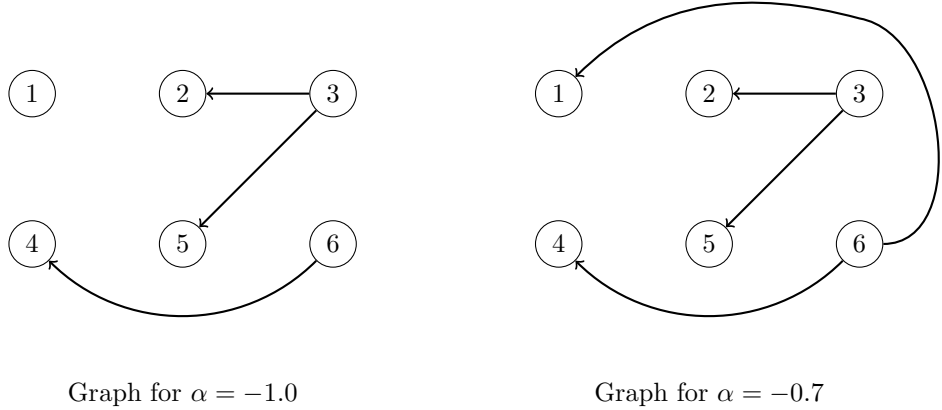
\begin{figure}
\centering
\begin{tikzpicture}
    \node[circle,draw] (1) at (0,0) {1};
    \node[circle,draw] (2) at (2,0) {2};
    \node[circle,draw] (3) at (4,0) {3};
    \node[circle,draw] (4) at (0,-2) {4};
    \node[circle,draw] (5) at (2,-2) {5};
    \node[circle,draw] (6) at (4,-2) {6};
    \draw[->,thick] (3) -- (2);
    \draw[->,thick] (3) -- (5);
    \draw[->,thick] (6) to [bend left=45] (4);
    \node at (2,-4) {Graph for $\alpha=-1.0$};
    \begin{scope}[xshift=7cm]
    \node[circle,draw] (1) at (0,0) {1};
    \node[circle,draw] (2) at (2,0) {2};
    \node[circle,draw] (3) at (4,0) {3};
    \node[circle,draw] (4) at (0,-2) {4};
    \node[circle,draw] (5) at (2,-2) {5};
    \node[circle,draw] (6) at (4,-2) {6};
    \draw[->,thick] (3) -- (2);
    \draw[->,thick] (3) -- (5);
    \draw[->,thick] (6) to [bend left=45] (4);
    \draw[->,thick] (6) to[out=0,in=-10] (4,1) to[out=165,in=45] (1);
    \node at (2,-4) {Graph for $\alpha=-0.7$};
    \end{scope}
\end{tikzpicture}
    \caption{Corresponding graphs of \cref{ex:illu_algo}}
    \label{fig:graph_illu}
\end{figure}

Such a result clearly goes against trying to find the maximin approach rather than the minimax one, at least in terms of computability and with the algorithms we have presented. However, if $k$ remains of low value (which we expect to be the case in most applications of a budgeted rule), the computational burden will remain limited. Also, we will see in the experiments that the maximin approach is very often (but not always) strongly consistent with $D_M$ and the maximality decision rule in practice, which is not the case for the minimax one. Before that, we will show that the maximin rule is also always weakly consistent with $D_M$

\subsection{Weak Consistency of maximin criterion}

Similarly to $S^*_k$, we will show that $S^+_k$ is weakly consistent with maximality. Weak consistency for $A$ such that $MmL^+_k(A)<0$ is ensured by \cref{cor:MmL:consistency:negative}, therefore we now look at the case $MmL^+_k(A) \geq 0$.

\begin{theorem}\label{lem:weak_consist_reg}
For any $S\subseteq A$ with cardinality $k$, there is a $T\subseteq A$ with cardinality $k$ such that $T \cap  D_M(A)\neq\emptyset$ and $MmL(T,A) \leq MmL(S,A)$.
\end{theorem}
\begin{proof}
If $S \cap D_M(A) \neq \emptyset$ then we can take $T=S$ and we are done.
For the remainder of the proof, we can therefore assume that $S \cap D_M(A)=\emptyset$.
It suffices to construct a subset $T$ of $A$ with cardinality $k$ such that 
$T \cap  D_M(A)\neq\emptyset$  and $MmL(T,A) \leq MmL(S,A)$, establishing the desired result.

For brevity, define $S'\coloneqq A\setminus S$. If $S \cap D_M(A)=\emptyset$, then $MmL(S,A) \ge 0$ by \cref{thm:MmL:negative}. There are $a_+ \in S$ and $a'_+ \in S'$ such that 
\begin{align}
0 \le MmL(S,A)=\uexpe(a'_+-a_+)
\end{align}
Since $S \cap D_M(A) = \emptyset$, and $a_+\in S$,
we know that $a_+$ is non-maximal in $A$.
Therefore, $a_+$ must be dominated by a maximal element of $A$, which is by assumption guaranteed to belong to $S'$.
In other words, there is a $b \in S'\cap D_M(A)$ such that
\begin{align}\label{eq:b_dominating_aplus}
\uexpe(a_+ - b) <0
\end{align}

Define $T\coloneqq (S\setminus \{a_+\}) \cup \{b\}$.
Clearly, $T$ has the same cardinality as $S$, since $a_+\in S$ and $b\not\in S$ by construction,
and $T\cap D_M(A)\neq\emptyset$ since $b\in D_M(A)$.
Since $MmL(S,A)=\uexpe(a'_+-a_+)$, also by construction,
it suffices to show that
\begin{equation}
MmL(T,A)=\max_{a'\in T'}\min_{a\in T}\uexpe(a'-a)
\le
\uexpe(a'_+-a_+)
\end{equation}
to finish our proof. In other words,
we are left to prove that
\begin{align}
\forall a'\in T',\,\exists a\in T\colon \uexpe(a' - a) \leq \uexpe(a'_+-a_+).
\end{align}

Indeed, by the definition of $MmL(S,A)$ we already know that
\begin{align}
\label{eq:a_j+i+_greater}
\forall a' \in S', \exists a \in S, &\ \uexpe(a'-a) \leq \uexpe(a'_+-a_+)
\end{align}
First note that we have $\lexpe(b-a_+)>0$, due to \cref{eq:b_dominating_aplus} and the duality relation $\lexpe(f)=-\uexpe(-f)$. From this we can deduce
\begin{equation}
\uexpe(b-a)
\ge 
\lexpe(b-a_+)+\uexpe(a_+-a)
>
\uexpe(a_+-a)
\end{equation}
so when $a'=b$ in \cref{eq:a_j+i+_greater},
we can replace $a'$ with $a_+$ whilst still
respecting the inequality. 
Therefore,
\begin{align}
\forall a'\in T',\,\exists a\in S\colon \uexpe(a' - a) \leq \uexpe(a'_+-a_+).
\end{align}

Now fix any $a'\in T'$.
If the $a\in S$ in the above condition is not equal to $a_+$, then $a\in T$ and we are done.
Otherwise, if $a=a_+$, note that again we can replace it with $b$ whilst still respecting the inequality, because, since $\uexpe(a_+-b)>0$,
\begin{equation}
\uexpe(a' - b)
\le
\uexpe(a' - a_+)+\uexpe(a_+ - b)
<
\uexpe(a' - a_+)
\end{equation}
This concludes the proof.
\end{proof}

\Cref{lem:weak_consist_reg} tells us that whenever we have a set of $k$ elements without any maximal ones, we can add a maximal ones and improve the solution by decreasing its $MmL$ value. This is sufficient to prove that at least one set $S^+_k(A)$ contains a maximal element. Consequently, we have the following corollaries from \cref{lem:weak_consist_reg} and \cref{cor:MmL:consistency:negative}. 

\begin{corollary}\label{cor:S_^+k_weekly}
$S^+_k$ and $D^k_{MmL}$ are weakly consistent with $D_M$. 
\end{corollary}

\begin{corollary}\label{cor:S^+_k_strongly}
$S^+_1$ and $D^1_{MmL}$ are strongly consistent with $D_M$. 
\end{corollary}

After investigating the properties of our two budgeted regret-based rules, as well as the computational methods to obtain them, we are ready to proceed to some experiments to test some of their behaviours. 

\section{Computational experimentation}\label{sec:computation}
In this section, we will perform some experiments to compare $S^*_k$ and $S^+_k$ with their greedy approximations $S^{g*}_k$ and $S^{g+}_k$, obtained as follows: given a set of actions $A$, we compute $S^{g*}_1(A) = S^*_1(A)$. Then, for $k \geq 2$, we compute $S^{g*}_k(A) =  S^{g*}_{k-1}(A) \cup S^*_1(A\setminus S^{g*}_{k-1}(A))$ and determine $mML(S^{g*}_k,A)$.  Similarly, we compute $S^{g+}_1(A) = S^+_1(A)$ and determine $MmL(S^{g+}_k,A)$. Then, for $k \geq 2$, we can compute $S^{g+}_k(A) =  S^{g+}_{k-1}(A) \cup S^+_1(A\setminus S^{g+}_{k-1}(A))$ and determine $mML(S^{g+}_k,A)$. Recall that the greedy algorithms are not optimal and will usually not result in $S^*_k$ and $S^+_k$, as already shown in \cref{ex:mML,ex:MmL2}.

We aim to compare $S^*_k$ and $S^+_k$ with their greedy approach for consistency with respect to maximality. Basically, we want to 
find out how much $S^*_k$, $S^{g*}_k$, $S^+_k$ and $S^{g+}_k$ can capture maximal elements in the set $D_M$. This can measure the quality of the greedy approximations with our proposed algorithms to find $S^*_k$ and $S^+_k$. Since $S^*_1$ and $S^+_1$ are weakly and strongly consistent with $D_M$ by \cref{cor:S_k_weekly,cor:S_k_strongly}  and \cref{cor:S_^+k_weekly,cor:S^+_k_strongly} respectively, we will not consider case $k = 1$. 

We fix $|A| = 20$, $|\Omega| = 5$, $|\credal|$ =20. We generate $p \in \credal$ by sampling a probability mass function $p$ uniformly from the unit simplex as follows. For each $\omega$, we sample $q(\omega)$ uniformly from (0,1) and then for each $\omega$, we assign $p(\omega) \coloneqq \frac{\ln q(\omega)}{\sum_{\omega}\ln q(\omega)}$. This ensures that generated distributions $p(\omega_1), \dots , p(\omega_n)$ follow a Dirichlet $\mathrm{Dir}(1, \dots, 1)$ distribution which has uniform density over the unit simplex \cite{feng2010Dirichlet}. Note that the convex hull of $\credal$ may have less than twenty extreme points, as some generated probabilities may be convex combinations of the others. Next, 
we generate a set of elements $A$ on $\Omega$ for which $|D_M(A)| = 6$ by using algorithm 6 in \cite{2019:Nak:Troff:Cai}. We will compute $S^*_k$, $S^{g*}_k$, $S^+_k$ and $S^{g+}_k$ for the cases  
$k \in \{2, \dots, 6\}$. To do so, for all $a_i \neq a_j \in A$, we compute $\uexpe(a_j - a_i)$ with respect to the credal set $\credal$ as an input to our algorithms and compute $S^*_k$, $S^{g*}_k$, $S^+_k$ and $S^{g+}_k$ with respect to $A$. Next, we check whether $S^*_k$, $S^{g*}_k$, $S^+_k$ and $S^{g+}_k$ are weakly consistent (having only one maximal element in the set) or strongly consistent (having all elements in the set being maximal) with $D_M$ or not. We also compute the proportion of elements in $S^*_k$, $S^{g*}_k$, $S^+_k$ and $S^{g+}_k$ that are in $D_M$. We repeat this process 500 times and summarise the result in \cref{table:1}. The percentages of these sets
that satisfy weakly and strongly consistent properties and the average percentages of elements in these sets which are in $D_M$ are showed in the 3$^\text{rd}$-5$^\text{th}$ columns of \cref{table:1}.

\begin{table}
\centering
\begin{tabular}{|c|c|c|c|c|c|c|} \hline 
$D^k$ & $k$ & w.c. & s.c. & $\dfrac{|D^k \cap D_M|}{ |D^k|}$ & & \\ \hline
\multirow{5}{*}{$S^*_k$} & 2 & 100\% & 100\% & 100\% &\multirow{5}{*}{$S^*_k=S^{g*}_k$} & \multirow{5}{*}{$\dfrac{|S^{g*}_k \cap S^*_k|}{ |S^{g*}_k|}$} \\ \cline{2-5} 
& 3 & 100\% & 90.2\% & 96.7\% & &  \\ \cline{2-5} 
& 4  & 100\%  & 73.8\% & 92.7\% & &  \\ \cline{2-5} 
& 5 & 100\% & 57.0\%  & 89.8\% & &  \\\cline{2-5} 
& 6  & 100\%  & 44.0\% & 88.4\% & &  \\ \hline
\multirow{5}{*}{$S^{g*}_k$} 
& 2 & 100\%  & 94.8\% & 97.4\% & 32.2\% & 64.3\% \\ \cline{2-7} 
& 3 &  100\% & 78.6\% &92.5\%& 13.0\% & 61.4\%  \\\cline{2-7} 
& 4 & 100\%  & 58.2\% & 87.8\%& 9.2\%& 65.9\% \\ \cline{2-7} 
& 5 & 100\%  & 37.0\% & 82.9\% & 7.0\% &69.4\% \\\cline{2-7} 
& 6 & 100\%  & 15.8\% & 77.9\% & 9.4\% & 73.6\% \\ \hline  
\multirow{5}{*}{$S^+_k$} & 2 & 100\% & 99.2\%& 99.6\% &\multirow{5}{*}{$S^+_k=S^{g+}_k$} & \multirow{5}{*}{$\dfrac{|S^{g+}_k \cap S^+_k|}{ |S^{g+}_k|}$} \\ \cline{2-5} 
& 3 & 100\% & 99.8\% & 99.9\% & &  \\ \cline{2-5} 
& 4  & 100\%  & 100\% & 100\% & &  \\ \cline{2-5} 
& 5 & 100\% & 100\%  & 100\% & &  \\\cline{2-5} 
& 6  & 100\%  & 100\% & 100\% & &  \\ \hline
\multirow{5}{*}{$S^{g+}_k$} 
& 2 & 100\%  & 94.8\% & 97.4\%  & 32.6\% & 63.0\%  \\ \cline{2-7} 
& 3 &  100\% & 78.6\% & 92.5\%& 18.6\% & 65.8\%  \\\cline{2-7} 
& 4 & 100\%  & 58.2\% & 87.8\%& 16.4\%& 71.1\% \\ \cline{2-7} 
& 5 & 100\%  & 37.0\% & 82.9\% & 15.2\% & 74.9\% \\\cline{2-7} 
& 6 & 100\%  & 15.8\% & 77.9\% & 15.8\% & 77.9\% \\ \hline 
\end{tabular}
\vspace{0.5em}
\caption{Percentages averages of $S^*_k$, $S^{g*}_k$, $S^+_k$ and $S^{g+}_k$ that satisfy different conditions.}\label{table:1}
\end{table}

According to the results, we see that $S^*_k$ and $S^+_k$ and their greedy approach are weakly consistent with $D_M$. Interestingly, only $S^+_k$ is likely to be strongly consistent while the rest of the sets are rarely strongly consistent with $D_M$ (the numbers in the fourth column quickly drop for all sets as $k$ increase, while they actually increase for $S^+_k$). Moreover, the average percentages of maximal elements in $S^*_k$ and $S^+_k$ are higher than in their greedy approximations $S^{g*}_k$ and $S^{g+}_k$. 

By \cref{thm:S^+=S^*} and the procedure of construct $S^{g*}_k$ and $S^{g+}_k$, we have $S^{g*}_k = S^{g+}_k$ for all $k$. Therefore, the percentages of $S^{g*}_k$ and $S^{g+}_k$ that are weakly and strongly consistent with $D_M$ are equal. In addition, to see how close those greedy approximations are to $S^*_k$ and $S^+_k$, we compare the optimal solutions with the greedy approach solutions. To do so, we record the number of $S^*_k = S^{g*}_k$ and the number of $S^+_k = S^{g+}_k$ and present the averages percentages of these sets that satisfy these conditions in the column 6$^\text{th}$. In the column 7$^\text{th}$, we calculate the average of proportion of elements in $S^{g*}_k$ that are in $S^*_k$ and and the proportion of elements in $S^{g+}_k$ that are in $S^+_k$. 

As we can employ $mML$ and $MmL$ as budgeted decision rules as in \cref{eq:mML,eq:MmL}, we want to find out how fast  $S^*_k$ and $S^+_k$ will become supersets of $D_M$ (or $mML^*_k(A)<0$ and $MmL^+_k(A)<0$) so that we can simply return $D_M$ instead of $S^*_k$ or $S^+_k$. To do so, we regenerate a set of elements $A$ on $\Omega$ such that $|D_M|=m$ for $m \in \{2,5,10\}$. Next, we compute $mML^*_k(A)$ and $MmL^+_k(A)$ for $k = m + i$, where $i \in \{0,1,2,3\}$. We repeat this process 100 times and present the result in \cref{table:2}. The average percentages of $mML^*_k(A)<0$ and $MmL^+_k(A)<0$ are presented in the 2$^\text{nd}$ and 3$^\text{rd}$ columns while the average percentages of $mML^*_k(A) = MmL^+_k(A)$ is presented in the 4$^\text{th}$ column of \cref{table:2}. According to the result, we found that $S^+_k$ becomes a superset of $D_M$ much faster than $S^*_k$ as the average percentages of $MmL^+_k(A)<0$ are much higher than the average percentages of $mML^*_k(A) <0$. If the cardinality of $D_M$ is increasing, then the average percentages of $mML^*_k(A)= MmL^+_k(A)$ tend to be decreasing. 

\begin{table}
\centering
\begin{tabular}{|c|c|c|c|c|}\hline
$|D_M|$ & $k$ & $mML^*_k(A)<0$ & $MmL^+_k(A)<0$ & $mML= MmL$ \\ \hline
\multirow{4}{*}{2} & 2 & 54\% & 100\% & 46\% \\ \cline{2-5} 
 & 3 & 76\% & 100\% & 51\% \\ \cline{2-5} 
& 4 & 89\% & 100\%  & 51\% \\ \cline{2-5} 
 & 5 & 96\% & 100\%  & 54\% \\ \hline
\multirow{4}{*}{5} & 5 & 33\% & 100\% & 23\%                 \\ \cline{2-5} 
& 6 & 64\% & 100\%  & 25\% \\ \cline{2-5} 
& 7 & 85\% & 100\%  & 31\% \\ \cline{2-5} 
& 8 & 96\% & 100\%  & 34\% \\ \hline
\multirow{4}{*}{10} & 10 & 28\% & 100\% & 15\%                 \\ \cline{2-5} 
& 11 & 66\% & 100\%  & 22\%\\ \cline{2-5} 
& 12 & 84\% & 100\%  & 19\% \\ \cline{2-5} 
& 13 & 97\% & 100\%  & 24\% \\ \hline
\end{tabular}
\vspace{0.5em}
\caption{Percentages averages of $S^*_k$ and $S^+_k$ that satisfy different conditions.}\label{table:2}
\end{table}

All those numbers show that, in practice, the maximin approach has a quite stronger consistency with maximality, and its negativity can be used as a quite reliable signal that we have captured all the maximal elements. In contrast, the minimax rule shows a much weaker consistency, and will often contain non-maximal elements. As maximality rests on very strong theoretical foundations, our conclusion is that the maximin rule should be preferred whenever its computational burden remains affordable, and that one should resort to the minimax rule only when computational efficiency is a key issue.

\section{Two illustrative use cases}
\label{sec:illustrations}

In this section, we demonstrate how our method can be applied in practice. The first example is inspired from \citet{jansen2022quantifying}, but adapted to provide more than $3$ maximal acts, while the second one concerns a situation where we must predict binary vectors over a set of labels, which is the situation encountered in multi-label learning, a specific multi-task machine learning problem. 

\subsection{Financial investment example}
We follow the financial application example in \cite{jansen2022quantifying}, where a subject wants to invest her money in some stocks. An act corresponds to investing in a specific stock. Suppose that a financial agent offers her ten different stocks, so we have $A=\{a_1, \dots, a_{10}\}$. There are five possible states of nature $\Omega=\{\omega_1, \dots, \omega_5\}$ corresponding to different economic scenarios that are uncertain to the agent.
The payoff for each stock under each possible scenarios is given in \cref{tab:financ_acts}.

\begin{table}
\centering
\begin{tabular}{c|ccccc}
 & $\omega_1$ & $\omega_2$ & $\omega_3$ & $\omega_4$ & $\omega_5$ \\ \hline
$a_1$ & 37 & 25 & 23 & 73 & 91 \\
$a_2$ & 50 & 67 &  2 & 44 & 94 \\
$a_3$ & 60& 4& 96& 1& 83 \\
$a_4$ & 16& 24& 31& 26& 100 \\
$a_5$ & 3&  86&  76&  85&  11\\ 
$a_6$ & 12&  49&  66&  56&  14 \\ 
$a_7$ & 39&  10&  92&  88&  57 \\ 
$a_8$ & 62&  52&  80&  71&  42 \\ 
$a_9$ & 90&  8&  74&  70&  38 \\ 
$a_{10}$ & 63&  68&  36& 69 &  9
\end{tabular}
\vspace{0.5em}
\caption{Payoffs for the acts of the financial application example from \cite{jansen2022quantifying}.}\label{tab:financ_acts}
\end{table}

In addition, based on the decision maker's experience in the financial market, she specifies her credal set through probability bounds as follows
(note these are slightly wider than the ones from \citet{jansen2022quantifying} to better demonstrate the benefits of our method):
\begin{equation}
\begin{aligned}
\credal = \{p \in \mathbb{P} \colon
& 0.1 \leq p(\omega_1) \leq 0.3, 0.05 \leq p(\omega_2) \leq 0.2, \\
& 0.1 \leq  p(\omega_3) \leq 0.2, 0.2 \leq  p(\omega_4) \leq 0.4, 0.1 \leq  p(\omega_5) \leq 0.3\}.
\end{aligned}
\end{equation}
From these bounds, the values of $\uexpe(a_j - a_i)$, for all $j \neq i$, can be calculated by linear programming. They are provided in \cref{table:example:finance}. 

\begin{table}
\centering
{\setlength{\tabcolsep}{2pt}  
\addtolength{\tabcolsep}{2pt}
\begin{tabular}{|c|*{10}{r}|}\hline
\backslashbox{$j$}{$i$}
&\makebox[1.5em]{1}&\makebox[1.52em]{2}&\makebox[1.5em]{3}
&\makebox[1.5em]{4}&\makebox[2em]{5}&\makebox[1.5em]{6} &\makebox[1.5em]{7} &\makebox[1.5em]{8} &\makebox[1.5em]{9} &\makebox[1.5em]{10}\\\hline
1 & - & 11.65 & 25.0 & 23.5 & 22.85 & 29.35 & 2.9 & 4.7 & 9.8 & 19.5 \\
  2 &5.0 & - & 21.6 & 20.7 & 21.5 & 28.4 & 6.9 & 3.0 & 9.6 & 14.3 \\
  3 &4.05 & 7.3 & - & 15.95 & 20.8 & 26.35 & $-3.0$ & $-1.4$ & 0.3 & 16.65 \\
  4 & $-7.4$ & $-4.25 $& 6.5 & - & 8.95 & 15.0 & $-7.4$ & $-11.2$ & $-3.0$ & 4.35 \\
  5 &16.2 & 22.0 & 33.1 & 34.8 & - & 19.8 & 2.6 & 2.6 & 10.6 & 12.2 \\
  6 &$-3.6$ & 2.2 & 13.3 & 15.0 & $-5.55$ & - & $-16.2$ & $-17.2$ & $-9.2$ & $-3.1$ \\
  7 &16.15 & 26.3 & 30.5 & 37.75 & 23.8 & 30.4 & - & 8.7 & 10.6 & 26.95 \\
  8 &19.0 & 23.45 & 32.3 & 34.95 & 23.1 & 28.7 & 8.0 & - & 8.6 & 18.5 \\
  9 &18.9 & 26.95 & 30.3 & 39.25 & 27.0 & 32.85 & 5.8 & 4.6 & - & 20.25 \\
  10 &10.0 & 11.6 & 27.2 & 27.2 & 8.5 & 19.5 & 2.7 & $-4.8$ & $-0.5$ & - \\ \hline
\end{tabular}
}
\vspace{0.5em}
\caption{Values of $\uexpe(a_j - a_i)$.}\label{table:example:finance}
\end{table}	

We see that while our information can discard some items as being non-maximal, we still have $\{a_1,a_2,a_5,a_7, a_8, a_9\}$ for the set of maximal acts, which may be judged too high if these represent complex financial portfolios. Results of our algorithms are shown in \cref{table:resutls_S*_S+}. In this case, the two approaches only differ slightly, and are completely consistent with the notion of maximality, as all selected examples for $k\leq 6$ are maximal. 

\begin{table}
\centering
{
\begin{tabular}{|l|l|c|l|c|}\hline
$k$ & $S^*_k(A)$ & $mML^*_k(A)$ &  $S^+_k(A)$ & $MmL^+_k(A)$ \\ \hline
1 & $a_7$ & 8 & $a_7$ &  8  \\ \hline
2 & $a_7, a_8$ & 4.7 & $a_7, a_8$ & 4.6 \\ \hline
3 & $a_1, a_7, a_8$ & 4.6 &  $a_7, a_8, a_9$ & 3 \\ \hline
4 &  $a_2, a_7, a_8, a_9$ & 2.9 & $a_2, a_7, a_8, a_9$  & 2.9\\ \hline
5 & $a_1, a_2, a_7, a_8, a_9$ & 2.6 & $a_1, a_2, a_7, a_8, a_9$ & 2.6 \\ \hline
6 & $a_1, a_2, a_5, a_7, a_8, a_9$ & $-1.4$ &  $a_1, a_2, a_5, a_7, a_8, a_9$ & $-3.0$ \\  \hline
\end{tabular}
}
\vspace{0.5em}
\caption{$S^*_k$ and $S^+_k$ for different values of $k$.}\label{table:resutls_S*_S+}
\end{table}

\subsection{Multi-label example}

To demonstrate our approach in machine learning, we consider multi-label classification~\cite{liu2021emerging}, a sub-category of multi-task learning where one first observes an input $x$ and then has to predict binary vectors over a set of labels $\bm{\ell}\coloneqq(\ell_1,\dots,\ell_n)$ with $\ell_i \in \{0,1\}$. A value $\ell_i=1$ usually means that the label is present in the instance $x$, while a zero means that the label is absent. There are many learning schemes to solve this problem, including imprecise ones~\cite{antonucci2017multilabel,moral2022partial,alarcon2021multi}, as well as many commonly used loss functions.  

\begin{table}
\begin{displaymath}
\begin{array}{c|cccccccc}
   &  [0 0 0] &   [0 0 1] &   [0 1 0] &   [0 1 1] &   [1 0 0] &   [1 0 1] &   [1 1 0] &   [1 1 1] \\
\hline
      \left[0 0 0\right]   & 0 &        -1 &        -1 &        -2 &        -1 &        -2 &        -2 &        -3 \\
      \left[0 0 1\right]   & -1 &         0 &        -2 &        -1 &        -2 &        -1 &        -3 &        -2 \\
      \left[0 1 0\right]   & -1 &        -2 &         0 &        -1 &        -2 &        -3 &        -1 &        -2 \\
       \left[0 1 1\right]  & -2 &        -1 &        -1 &         0 &        -3 &        -2 &        -2 &        -1 \\
      \left[1 0 0\right]   & -1 &        -2 &        -2 &        -3 &         0 &        -1 &        -1 &        -2 \\
      \left[1 0 1\right]   & -2 &        -1 &        -3 &        -2 &        -1 &         0 &        -2 &        -1 \\
      \left[1 1 0\right]   & -2 &        -3 &        -1 &        -2 &        -1 &        -2 &         0 &        -1 \\
     \left[1 1 1\right]    & -3 &        -2 &        -2 &        -1 &        -2 &        -1 &        -1 &         0 \\
\hline
\end{array}
\end{displaymath}
\caption{Example of multi-label utility based on Hamming distance for $n=3$.}
\label{tab:multi_label_utility}
\end{table}

As is classically done in machine learning, our sets of acts and of states will coincide, acts being prediction of the possible ground truth when observing $x$. We will therefore have $\Omega=\mathcal{A}=\{0,1\}^n$, with an exponentially increasing size of sets of alternatives as a function of the number of labels $n$. We will also consider here the utility version of the standard Hamming loss function, which means that if $\bm{\hat{\ell}}$ is the predicted vector (i.e., the act, in our setting), and $\bm{\ell}$ the true one (i.e., the state, in our setting), then 
$$\bm{\hat{\ell}}(\bm{\ell})=-\sum_{i=0}^n \mathbb{I}_{\hat{\ell}_i \neq \ell_i},$$
where $\mathbb{I}_A$ denotes the indicator function of event $A$. \Cref{tab:multi_label_utility} illustrates the obtained utility matrix when $n=3$. As a model, we will consider in this example an imprecise version of the classical binary relevance scheme, where the probability mass of a given vector $\ell$ is the product of label-wise probability masses, i.e.,
\begin{equation}\label{eq:joint_multi_label}p(\bm{\ell})=\prod_{i=1}^n p_i(\ell_i)\end{equation}
Here, the $p_i$'s values are outputted by classifiers. We will consider here that instead of having precise label-wise estimates, we have imprecise ones given as follows:
$$p_1(\ell_1=1) \in [0.4,0.8], \quad p_2(\ell_2=1) \in [0.2,0.6], \quad p_3(\ell_3=1) \in [0.1,0.7].$$
By robustifying the product in \cref{eq:joint_multi_label}, we get a set of extreme probabilities obtained by considering all the combinations of interval bounds. Those are summarised in \cref{tab:ext_prob_mult_lab}. 

\begin{table}\small
\begin{displaymath}
    \begin{array}{c|cccccccc}
   & p([0 0 0]) &   p([0 0 1]) &   p([0 1 0]) &   p([0 1 1]) &   p([1 0 0]) &   p([1 0 1]) &   p([1 1 0]) &   p([1 1 1]) \\
\hline
   p_1 & 0.432 &     0.048 &     0.108 &     0.012 &     0.288 &     0.032 &     0.072 &     0.008 \\
   p_2 & 0.144 &     0.336 &     0.036 &     0.084 &     0.096 &     0.224 &      0.024 &     0.056 \\
   p_3 & 0.216 &     0.024 &     0.324 &     0.036 &     0.144 &     0.016 &     0.216 &     0.024 \\
   p_4 & 0.072 &     0.168 &     0.108 &     0.252 &     0.048 &     0.112 &     0.072 &     0.168 \\
   p_5 & 0.144 &     0.016 &     0.036 &     0.004 &     0.576 &     0.064 &     0.144 &     0.016 \\
   p_6 & 0.048 &     0.112 &     0.012 &     0.028 &     0.192 &     0.448 &     0.048 &     0.112 \\
   p_7 & 0.072 &     0.008 &     0.108 &     0.012 &     0.288 &     0.032 &     0.432 &     0.048 \\
   p_8 & 0.024 &     0.056 &     0.036 &     0.084 &     0.096 &     0.224 &     0.144 &     0.336 \\
\hline
\end{array}
\end{displaymath}
    \caption{Extreme probabilities of multi-label example.}
    \label{tab:ext_prob_mult_lab}
\end{table}

All is left to do now is to compute the pairwise matrix of upper expectations, which is summarised in \cref{tab:upper_mult_lab}. From this table it is clear that all alternatives are maximal, hence using a robust decision rule such as maximality is not helpful at all to select an alternative. In general, we would expect the number of maximal vectors to be quite large but not equal to $\Omega$. Also note that due to the simplicity of the model and the highly structured form of the utility matrix in \cref{tab:multi_label_utility}, upper expectations only take a limited number of values. This would however not be the case in actual applications, where an analyst would probably consider more complex models, as well as utility or loss functions that would be different for different label mistakes. 

Assuming that we want to return only two vectors, let us first consider the minimax criterion. Applying \cref{alg:S^*_k} to \cref{tab:upper_mult_lab}, we get that the minimal second largest element is obtained for column $[100]$ (the only one for which it is below 1), with $S^*_2=\{[100],[011]\}$ and $mML(S^*_2)=0.6$. Considering now the maximin criterion, if we fix $\alpha=0.4$ (from the $mML$ value, we already know it is equal or lower than 0.6), we get $C_{0.4}([100])=\{[000],[010],[101],[110]\}$ and $C_{0.4}([101])=\{[001],[011],[111]\}$, that do provide a dominating set, and one can check $S^+_2=\{[100],[101]\}$ and $MmL(S^+_2)=0.4$. Despite the high structure of the problem, the two approaches deliver distinct, unique solutions. They also seem to adopt different strategies, the minimax recommending very diverse, complementary vectors, while the maximin sends back two similar vectors, that differ only by the label that is the most uncertain ($\ell_3$). 

\begin{table}
\begin{displaymath}
    \begin{array}{c|cccccccc}
   & [0 0 0] &   [0 0 1] &   [0 1 0] &   [0 1 1] &   [1 0 0] &   [1 0 1] &   [1 1 0] &   [1 1 1] \\
\hline
      & 0.0 &       0.8 &       0.6 &       1.4 &       0.2 &       1.0 &       0.8 &       1.6 \\
      & 0.4 &       0.0 &       1.0 &       0.6 &       0.6 &       0.2 &       1.2 &       0.8 \\
      & 0.2 &       1.0 &       0.0 &       0.8 &       0.4 &       1.2 &       0.2 &       1.0 \\
      & 0.6 &       0.2 &       0.4 &       0.0 &       0.8 &       0.4 &       0.6 &       0.2 \\
      & 0.6 &       1.4 &       1.2 &       2.0 &       0.0 &       0.8 &       0.6 &       1.4 \\
      & 1.0 &       0.6 &       1.6 &       1.2 &       0.4 &       0.0 &       1.0 &       0.6 \\
      & 0.8 &       1.6 &       0.6 &       1.4 &       0.2 &       1.0 &       0.0 &       0.8 \\
      & 1.2 &       0.8 &       1.0 &       0.6 &       0.6 &       0.2 &       0.4 &       0.0 \\
\hline

\end{array}
\end{displaymath}
    \caption{Upper pairwise expectations for the multi-label example.}
    \label{tab:upper_mult_lab}
\end{table}

\section{Discussion and conclusion}\label{sec:conclusion}
In this study, we have studied $k$-budgeted regret-based decision rules that return an optimal subset of size $k$ with respect to some value function. To do so, we recalled minimax criteria which minimizes the maximal gain to the adversary on a given set of alternatives and proposed a new regret-based decision rule called maximin criteria which swaps the order of selecting alternatives by the decision maker and the adversary. We also provided algorithms for both criteria and discussed their properties with respect to maximality. Note that our framework can be extended straightforwardly to continuous spaces as long as one can compute upper expectations over it, while extending it to an infinite set of acts would be trickier.

From the experimental perspective, we compared the minimax and maximin criteria and their greedy approximation on generated sets. Overall, both algorithms perform better than their greedy approximation. We also observed that the maximin criteria can capture all maximal elements faster than the minimax criteria and is more consistent with the maximality decision rule. However, the computational complexity of our proposed algorithm for maximin criteria is much harder than for the minimax criteria, as we showed that the former requires to solve an NP-complete problem, while the latter can be solved in polynomial time. This drawback is however of limited importance for small values of $k$. So, one should clearly prefer the maximin approach as long as it is computationally affordable, and otherwise take the minimax approach.

Note that whilst we are not aware of other works considering the specific problem of delivering at most $k$ alternatives when modelling uncertainty as a credal set, some works provide methods and ideas that could easily be leveraged to do so. A first strand of works in this direction would consist in considering nested models where the imprecision is controlled by some parameter, allowing one to go from a precise probability mass function (and having a unique optimal action) to the full credal set. 
\Citet{jansen2018concepts} consider such a parametric model for imprecisely defined utility functions (that we assume here well-defined), further discussed by \citet{miranda2023inner} in combination with credal set approximation. Those same authors also discuss the notion of centroid for credal sets~\cite{miranda2023centroids}, and it would be easy to go from this to the idea of a nested sequence. A second proposal is to associate an evaluation to each possible maximal and/or E-admissible (a concept we did not consider here) alternatives~\cite{jansen2022quantifying}, to rank-order them according to this evaluation and to take the top-k alternatives. Such approaches also appear legitimate to solve the issue considered in this paper, with a different philosophy. In particular, they would consider alternatives individually, rather than offering an evaluation for a whole set of alternatives. One risk is then that potential interactions between alternatives would be ignored, and that similar alternatives could be selected in the $k$ retained ones (or, in the case of E-admissibility degrees, that none of a collection of similar alternatives would be retained, while it would probably be desirable to retain at least one of those). On the other hand, one advantage over our approach is that such approaches makes it easier to be coherent with existing decision rules, and some of them can be solved efficiently. Finally, we also think that regret formulations are well-known and appealing to many researchers, which is also an argument to develop such methods in addition to others. One of our future endeavour would be to compare those various approaches, both from theoretical and practical perspectives, yet this would require to discuss desirable properties of credal budgeted decision rules, something that is out of the scope of the current paper.

In other future work, we can look at a more practical perspective, for example, we may apply these proposed budgeted rules to actual decision problems such as machine learning with structured outputs or system design, where the decision maker is limited by human cognitive limits or where inspecting more closely the different proposed options can lead to a high monetary cost.

\section*{Acknowledgement}
This project is funded by National Research Council of Thailand (NRCT). This research was supported by Chiang Mai University. NN would like to thank Assoc. Prof. Manad Khomkong for his support. We also thank Tom Davot for discussions about the complexity proof. 

\bibliographystyle{plainnat} 
\bibliography{cas-refs}

\end{document}